 \documentclass[preprint,review,12pt]{elsarticle}

\usepackage{amssymb}
 \usepackage{amsthm,amsmath}
\usepackage{mathrsfs}

 \numberwithin{equation}{section}

\newtheorem{theorem}{Theorem}[section]
\newtheorem{lemma}{Lemma}[section]
\newtheorem{remark}{Remark}[section]

\newtheorem{definition}{Definition}[section]
\newcommand{\beq}{\begin{eqnarray}}
\newcommand{\eeq}{\end{eqnarray}}
\newcommand{\beqno}{\begin{eqnarray*}}
\newcommand{\eeqno}{\end{eqnarray*}}

\newenvironment {Proof of Theorem} {\noindent {\bf Proof of Theorem 2.1}}{\quad $\square$\par\vspace{3mm}}

\allowdisplaybreaks

\journal{SCIENCE CHINA Mathematics}

\begin{document}
\newcommand{\D}{\displaystyle}
\begin{frontmatter}



\title{Global well-posedness for the dynamical $Q$-tensor model of liquid crystals\emph{}}


\author{Jinrui Huang$^{1}$}
\author{Shijin Ding$^{2*}$\corref{cor1}}
\cortext[cor2]{Corresponding author. Email: dingsj@scnu.edu.cn}
\address{$^{1}$School of Mathematics and Computational Science, Wuyi University, Jiangmen 529020, P.R. China}
\address{$^{2}$School of Mathematical Sciences, South China Normal University,  Guangzhou  510631,  P.R. China}

\begin{abstract}
In this paper, we consider a complex fluid modeling nematic liquid crystal flows, which is described by a system coupling Navier-Stokes equations with a parabolic $Q$-tensor system. We first prove the global existence of weak solutions in dimension three. Furthermore, the global well-posedness of strong solutions is studied with sufficiently large viscosity of fluid. Finally, we show a continuous dependence result on the initial data which directly yields the weak-strong uniqueness of solutions.
\end{abstract}


\end{frontmatter}

\newpage
\setcounter{section}{0} \setcounter{equation}{0}
\section{Introduction}
The most comprehensive continuum theory for nematic liquid crystal flows is the $Q$-tensor theory proposed by P.G. de Gennes \cite{gennes}, in which the dynamics of nematic liquid crystals is modeled by the Navier-Stokes equations coupled with a parabolic equation of $Q$-tensor. $Q$-tensor is the order parameter which is symmetric and traceless $3\times3$ matrix, and can be viewed as the second order moment of the orientational distribution function $f(x,m,t)$ \cite{wang-zhang-zhang}:
$$Q=\int_{S^2}\left(m\otimes m-\frac13I_3\right)f{\rm d}m.$$

The nematic liquid crystal can be classified accordingly. If $Q=0$, we say that the liquid crystal is isotropic. If $Q$ admits two equal non-zero eigenvalues, it is called uniaxial. If $Q$ has three distinct eigenvalues, it is said to be biaxial.

In uniaxial case (see \cite{ding-huang-lin,huang-ding,huang-lin-wang} and their references for details), $Q$ can be written in the special form
$$Q=s\left(n\otimes n-\frac13I_3\right), \ s\in\mathbb{R}-\{0\}, \ n\in S^2.$$

If the liquid crystal is biaxial, $Q$-tensor can be represented in the special form  \cite{Majumdar}:
$$Q=s\left(n\otimes n-\frac13I_3\right)+r\left(m\otimes m-\frac13I_3\right), \ s,r\in\mathbb{R} \ {\rm and}\ n,m\in S^2.$$

Landau-de Gennes theory gives us the following energy functional which consists of the elastic energy and the bulk energy:
\begin{eqnarray*}
    \mathcal{F}_{LG}[Q]=\int_{\Omega} f_{elasticity}(Q,\nabla Q)(x)+f_{bulk}(Q)(x){\rm d}x.
\end{eqnarray*}

The most widely-accepted forms of the elastic energy and the bulk energy densities are defined as follows (see \cite{Mottram}):
\begin{eqnarray*}
    f_{elasticity}(Q,\nabla Q)&=&\frac12\sum_{i,j,k=1}^3\left[L_1\left(\frac{\partial Q_{ij}}{\partial x_k}\right)^2+L_2\frac{\partial Q_{ij}}{\partial x_j}\frac{\partial Q_{ik}}{\partial x_k}+L_3\frac{\partial Q_{ik}}{\partial x_j}\frac{\partial Q_{ij}}{\partial x_k}\right]
    \\&&
    +\frac12\sum_{i,j,k,l}^3\left[L_4e_{lik}Q_{lj}\frac{\partial Q_{ij}}{\partial x_k}+L_5Q_{lk}\frac{\partial Q_{ij}}{\partial x_l}\frac{\partial Q_{ij}}{\partial x_k}\right]
    \\
    f_{bulk}(Q)&=&\frac{a}{2}{\rm tr}(Q^2)-\frac{b}{3}{\rm tr}(Q^3)+\frac{c}{4}\left({\rm tr}(Q^2)\right)^2,
\end{eqnarray*} where $L_i$ $(i=1,2,3,4,5)$ are elastic constants, $a,b,c$ are material-dependent and temperature-dependent constants. $e_{ijk}$ is the Levi-Civita symbol, i.e., $e_{ikj}$ is $1$ if $(i,j,k)$ is an even permutation of $(1,2,3)$, $-1$ if it is an odd permutation, and $0$ if any index is repeated. Throughout this paper, we will use the Einstein summation convention.

One of the most famous dynamical $Q$-tensor models is proposed by Beris-Edwards \cite{Beris,wang-zhang-zhang} (see also \cite{qian-shen}). Define
\begin{eqnarray*}
    H_Q=-\frac{\delta {\mathcal{F}_{LG}[Q]}}{\delta Q},
\end{eqnarray*} where we should note that this variation must be subject to the constraint that $Q$ is both symmetric and traceless.
Then the Beris-Edwards model can be written in the following form:
\begin{eqnarray} \label{model}
\begin{cases}\label{be}
      u_t+u\cdot\nabla u+\nabla P=\nabla\cdot (\sigma^s+\sigma^a+\sigma^d),
      \\
      Q_t+u\cdot\nabla Q=\Gamma H_Q+F(Q,D)+\Omega\cdot Q-Q\cdot \Omega,
      \\
      \nabla\cdot u=0,
\end{cases}
\end{eqnarray}
for $(x,t)\in \mathbb{R}^3\times (0,+\infty)$ with the initial conditions:
\begin{eqnarray}\label{ic}
   (u,Q)\big|_{t=0}=(u_0(x),Q_0(x)) \ {\rm{in}}\ \mathbb{R}^3.
\end{eqnarray}
Here $u:\mathbb{R}^3\times [0,+\infty)\rightarrow \mathbb{R}^3$ represents the velocity field of the fluid, $Q:\mathbb{R}^3\times [0,+\infty)\rightarrow S_0^{(3)}$ represents the order parameter, with $S_0^{(3)}\subseteq \mathbb{M}^{3\times3}$ representing the space of $Q$-tensors in three dimensions, i.e.,
$$S_0^{(3)}=\{Q\in \mathbb{M}^{3\times3}: Q_{ij}=Q_{ji}, {\rm tr}(Q)=0, i,j=1,2,3\},$$
$P(x,t):\mathbb{R}^3\times [0,+\infty)\rightarrow \mathbb{R}$ is the pressure function, and
\begin{eqnarray*}
    &&\sigma^d_{ij}=-\frac{\partial {\mathcal{F}}_{LG}[Q]}{\partial Q_{kl,j}}Q_{kl,i},\ \sigma^a=Q\cdot H_Q-H_Q\cdot Q,\ \sigma^s=\mu D-F(Q,H_Q),
    \\&&
    D_{ij}=\frac{u_{i,j}+u_{j,i}}{2},\ \Omega_{ij}=\frac{u_{i,j}-u_{j,i}}{2},\ \Gamma,\mu>0,
    \\&&
    F(Q,A)=\xi\left[\left(Q+\frac13I_3\right)\cdot A+A\cdot\left(Q+\frac13I_3\right)-2\left(Q+\frac13I_3\right)(A:Q)\right].
\end{eqnarray*}

To guarantee that the basic energy is bounded from below and to get the dissipation law, we assume, throughout  this paper, that
\begin{eqnarray}
    \label{elasicity constant}
    c>0,\ L_5=0,\ L_1>0,\ L_2+L_3\geq 0.
\end{eqnarray}

In fact, there are other ways to reach this aim. For example, in \cite{wang-zhang-zhang2}, the authors assume $L_1>0$, $L_1+L_2+L_3>0$ and $L_4=L_5=0$ which is sufficient to guarantee that the basic energy is bounded from below. While in \cite{Bauman}, Bauman-Park-Phillips considered two cases: (i)\ $L_1>0$, $L_2+L_3\geq 0$, or (ii)\ $L_1+L_2+L_3>0$, $L_2+L_3\leq 0$. For the second case, one only needs to observe that
\begin{eqnarray*}
    |\nabla Q|^2-|{\rm div}Q|^2-|{\rm curl} Q|^2=Q_{ij,k}Q_{ik,j}-Q_{ij,j}Q_{ik,k}
\end{eqnarray*} is a null Lagrangian and therefore, if we let
\begin{eqnarray*}
    f^\prime_{elasticity}(Q)=\frac{L_1+L_2+L_3}{2}|\nabla Q|^2-\frac{L_2+L_3}{2}|{\rm curl}Q|^2,
\end{eqnarray*} then $f_{elasticity}-f^\prime_{elasticity}$ is a null Lagrangian.

For the Beris-Edwards dynamical model (\ref{be}) under the initial condition (\ref{ic}), Paicu et al obtained several results for the Cauchy problems in $R^d$ ($d=2$ or $3$). In \cite{Paicu2}, if $\xi=0$ in the expression of $F(Q, A)$ which means that the molecules are such that they only tumble in a shear flow, but are not
aligned by such a flow,
Paicu and Zarnescu proved the existence of weak solutions, the existence of a Lyapunov functional for the smooth solutions and used cancelations that allow its existence to prove higher global regularity in dimension 2. The weak-strong uniqueness in dimension 2 was also given. In \cite{Paicu1}, the situation $\xi\not=0$ was considered and similar results were obtained.

Feireisl, Rocca, Schimperna and Zarnescu in \cite{F} considered the Cauchy problem for the model coupling with energy equation. They constructed global in time weak solutions for arbitrary physically relevant initial data.

As for the initial-boundary problems, we refer to \cite{Abel1,Abel2,F-R,F-R-R,G-R}.

We should note that all above papers for dynamical model only concerned the simple situation in which the elastic energy reads as
$$
\mathcal{F}_{LG}[Q]=\int_{\Omega} \frac{L_1}2|\nabla Q|^2+f_{bulk}(Q),
$$
that is, $L_2=L_3=L_4=L_5=0$.

In this paper, we study the Cauchy problem for more general situations that we only assume (\ref{elasicity constant}) holds. For simplicity, we also assume that
\begin{eqnarray}
    \xi=0.
\end{eqnarray}
We remark that we have to assume $L_2+L_3\geq 0$ in order to get the $L^2$ space-time estimates of $Q$ with second order derivative in space from the basic energy law.

In order to understand Beris-Edwards' dynamical model, we should carry out the above mentioned variation calculation for the Landau-de Gennes functional, i.e. ${\delta {F_{LG}[Q]}}/{\delta Q}$ , under the constraint  that $Q$ is in $S^3_0$. For this reason, we introduce the following functional with Lagrangian multipliers $\lambda$ and $\lambda_{ij}$ ($i,j=1, 2, 3$):
\begin{eqnarray*}
    \mathcal{L}^{\lambda_0,\lambda_{ij}}_{LG}[Q]=\mathcal{F}_{LG}[Q]+\lambda_0{\rm tr}Q+\lambda_{ij}(Q_{ij}-Q_{ji})
\end{eqnarray*}
so that
\begin{eqnarray}\label{hq}
    H_Q=-\frac{\delta\mathcal{L}^{\lambda_0,\lambda_{ij}}_{LG}[Q]}{\delta Q},
\end{eqnarray}
where the variation is carried out in ${M}^{3\times3}$ subject to no constraints.

Let $Q$ be any minimizer of $F_{LG}[Q]$ in $S^3_0$, then
\begin{eqnarray*}
\frac{\delta\mathcal{L}^{\lambda_0,\lambda_{ij}}_{LG}[Q]}{\delta Q}=0.
\end{eqnarray*}

Since
\begin{eqnarray}\label{var}
  \nonumber -\frac{\delta\mathcal{L}^{\lambda_0,\lambda_{ij}}_{LG}[Q]}{\delta Q}
  &=&
   L_1\Delta Q_{ij}+(L_2+L_3)Q_{ik,kj}+L_4e_{lik}Q_{lj,k}\nonumber\\
       && -a Q_{ij}+bQ_{ik}Q_{kj}-c{\rm tr}(Q^2)Q_{ij}-\lambda_0\delta_{ij}-(\lambda_{ij}-\lambda_{ji}),
\end{eqnarray}
we obtain
\begin{eqnarray}\label{eq}
       && L_1\Delta Q_{ij}+(L_2+L_3)Q_{ik,kj}+L_4e_{lik}Q_{lj,k}\nonumber\\
       && -a Q_{ij}+bQ_{ik}Q_{kj}-c{\rm tr}(Q^2)Q_{ij}-\lambda_0\delta_{ij}-(\lambda_{ij}-\lambda_{ji})\nonumber\\
    &&=0.
\end{eqnarray}

Taking trace on the both sides of (\ref{eq}) and noticing that trace$(\delta_{ij})=3$, we have
\begin{eqnarray}
\label{lambda_0}
   \lambda_0=\frac{(L_2+L_3)}3Q_{kp,kp}+\frac b3{\rm tr}(Q^2)+\frac{L_4}3e_{lpk}Q_{lp,k},
\end{eqnarray}

It follows from (\ref{eq}) that, if $i\neq j$, there holds
\begin{eqnarray}\label{lambdaij}
   \lambda_{ij}-\lambda_{ji}=L_1\Delta Q_{ij}+(L_2+L_3)Q_{ik,kj}+L_4e_{lik}Q_{lj,k}
    -a Q_{ij}+bQ_{ik}Q_{kj}-c{\rm tr}(Q^2)Q_{ij}
\end{eqnarray}
which means from the symmetry property, $Q_{ij}=Q_{ji}$, that for $i\neq j$
\begin{eqnarray}\label{lambdaji}
   \lambda_{ji}-\lambda_{ij}
     =L_1\Delta Q_{ij}+(L_2+L_3)Q_{jk,ki}+L_4e_{ljk}Q_{li,k}
    -a Q_{ij}+bQ_{ik}Q_{kj}-c{\rm tr}(Q^2)Q_{ij}.
\end{eqnarray}

Subtracting (\ref{lambdaji}) from (\ref{lambdaij}), one obtains
\begin{eqnarray}\label{lambda}
    \lambda_{ij}-\lambda_{ji}=\frac{L_2+L_3}{2}(Q_{ik,kj}-Q_{jk,ki})+\frac{L_4}{2}(e_{lik}Q_{lj,k}-e_{ljk}Q_{li,k}).
\end{eqnarray}

Finally, inserting (\ref{lambda_0}) and (\ref{lambda}) into (\ref{var}) and (\ref{hq}), one gets
\begin{eqnarray*}
   (H_Q)_{ij}&=&\underbrace{L_1\Delta Q_{ij}+\frac{L_2+L_3}{2}\left(Q_{ik,kj}+Q_{jk,ki}-\frac23\delta_{ij}Q_{kp,kp}\right)}_{(M_Q)_{ij}}
   \\&&
   \underbrace{+\frac{L_4}{2}\left(e_{lik}Q_{lj,k}+e_{ljk}Q_{li,k}-\frac23\delta_{ij}e_{lpk}Q_{lp,k}\right)}_{(E_Q)_{ij}}
    \\&&
    \underbrace{-a Q_{ij}+b\left(Q_{ik}Q_{kj}-\frac13\delta_{ij}|Q|^2\right)-c{\rm tr}(Q^2)Q_{ij}}_{(B_Q)_{ij}}.
\end{eqnarray*}

On the other hand, we have the first term of the stress tensor in the momentum equations as follows:
\begin{eqnarray*}
    \sigma^d_{ij}&=&-\frac{\partial \mathcal{F}_{LG}[Q]}{\partial Q_{kl,j}}Q_{kl,i}\\
       &=&-\left(L_1Q_{kl,i}Q_{kl,j}+L_2Q_{km,m}Q_{kj,i}+L_3Q_{kj,l}Q_{kl,i}+\frac{L_4}{2}e_{mkj}Q_{ml}Q_{kl,i}\right).
\end{eqnarray*}

Before continuing, we would like to have some more words here on the developments for the $Q$-tensor models modeling the nematic liquid crystals made by Wei Wang, Pingwen Zhang and Zhifei Zhang in recent years.

As we know, the Landau-de Gennes functional is defined from phenomenological theory to describe: the elastic energy of any distortion to the structure of the material, thermotropic energy (bulk energy) to dictates the
the preferred phase of the material and so on.

In the paper \cite{han}, Han-Luo-Wang-Zhang established,
based on Onsager's molecular theory and using Bingham closure, a systematic way of liquid crystal modeling to build connection between microscopic theory and macroscopic theory. A new Q-tensor theory which leads to liquid crystals with certain shape was proposed. Making uniaxial assumption, they can recover the Oseen-Frank theory from the derived Q-tensor theory, and the Oseen-Frank model coefficients can be examined. For further study on this topic, one can also refer to \cite{chen-zhang,xu-zhang}.

Subsequently, in the paper \cite{wang-zhang-zhang}, Wang-Zhang-Zhang, starting from Doi-Onsager equation for the liquid crystals and also by the Bingham closure, derived a new system of dynamical Q-tensor equations which is somewhat different from Beris-Edwards model. Then they derived the Ericksen-Leslie equation from the
new Q-tensor equation by taking the small Deborah number limit.

In \cite{wang-zhang-zhang2}, starting from Beris-Edwards system for the liquid crystal, the authors presented a rigorous derivation of Ericksen-Leslie system with general Ericksen stress and Leslie stress by using the Hilbert expansion method.

While in \cite{wang-zhang-zhang3}, the authors presented a rigorous derivation of the Ericksen-Leslie equation directly from the Doi-Onsager equation, the molecule model.

{\bf Now we are in a position to state our main results.}

Firs of all, we explain the assumptions and notations used throughout this paper.

\noindent{\bf Notations:}

\noindent
(1) For $p\ge 1$, denote by $L^p=L^p(\mathbb{R}^3)$ the $L^p$ space with the
norm $\|\cdot\|_{L^p}$. For $k\ge 1$ and $p\ge 1$, denote by $W^{k,p}=W^{k,p}(\mathbb{R}^3)$ the Sobolev space whose norm is $\|\cdot\|_{W^{k,p}}$, $H^k=W^{k,2}(\mathbb{R}^3)$.

\noindent
(2) $\mathcal{V}=$ closure of $\{u\in C^\infty_0(\mathbb{R}^3,\mathbb{R}^3): {\rm div}u=0\}$ in $L^2(\mathbb{R}^3,\mathbb{R}^3)$, $\mathcal{H}=$ closure of $\{u\in C^\infty_0(\mathbb{R}^3,\mathbb{R}^3): {\rm div}u=0\}$ in $H^1(\mathbb{R}^3,\mathbb{R}^3)$. Denote by $\langle\cdot,\cdot\rangle$ the inner product on $L^2$ space.

\noindent
(3) Denote by $|Q|\doteq \sqrt{{\rm tr}(Q^2)}=\sqrt{Q_{\alpha\beta}Q_{\alpha\beta}}$ the Frobenius norm of a matrix $Q$. For $Q_1,Q_2\in S_0^{(3)}$, denote $Q_1: Q_2={\rm tr}(Q_1Q_2)$. And $|\nabla Q|^2=\partial_\gamma Q_{\alpha\beta}\partial_\gamma Q_{\alpha\beta}=Q_{\alpha\beta,\gamma}Q_{\alpha\beta,\gamma}$. Finally, the divergence of a tensor is defined by $\nabla\cdot \sigma=\partial_j\sigma_{ij}$.

\begin{definition} A pair $(u,Q)$ is called a global weak solution of system (\ref{be}) with the initial conditions (\ref{ic}) where $u_0(x)\in \mathcal{V}, \ Q_0(x)\in H^1$, if $u\in L^\infty_{loc}([0,+\infty);\mathcal{V})\cap L^2_{loc}([0,+\infty);\mathcal{H})$, $Q\in L^\infty_{loc}([0,+\infty);H^1)\cap L^2_{loc}([0,+\infty);H^2)$, and furthermore, for any compactly supported $\varphi\in C^\infty([0,+\infty)\times \mathbb{R}^3;\mathbb{R}^3)$ with $\nabla\cdot\varphi=0$ and $\psi\in C^\infty([0,+\infty)\times \mathbb{R}^3;S_0^{3})$, one has
\begin{eqnarray*}
    &&\int_0^\infty\left(-\langle u,\varphi_t\rangle -\langle u\otimes u,\nabla\varphi\rangle +\mu\langle \nabla u,\nabla\varphi\rangle \right){\rm d}t-\langle u_0(x),\varphi(x,0)\rangle
    \\&=&
    \int_0^\infty\langle \frac{\partial \mathcal{F}_{LG}}{\partial \nabla Q}\odot \nabla Q-Q\cdot H_Q+H_Q\cdot Q,\nabla\varphi\rangle {\rm d}t,
\end{eqnarray*} and
\begin{eqnarray*}
    &&\int_0^\infty\left(-\langle Q,\psi_t\rangle -\langle Q\cdot u,\nabla\psi\rangle -\langle \Omega Q-Q\Omega,\psi\rangle \right){\rm d}t
    \\&=&\langle Q_0(x),\psi(x,0)\rangle +\Gamma\int_0^\infty\langle H_Q,\psi\rangle {\rm d}t.
\end{eqnarray*}
\end{definition}

Our first result is the following theorem about the global existence of weak solutions for problem (\ref{be})--(\ref{ic}).

\begin{theorem}\label{weak solutions} There exists a global weak solution $(u,Q)$ of system (\ref{be})--(\ref{ic}).
\end{theorem}

The second result is the following theorem concerning the global existence of strong solutions for problem (\ref{be})--(\ref{ic}) provided that the viscosity of the fluid is sufficiently large.

\begin{theorem}\label{th:1.1}
For any $(u_0,Q_0)\in \mathcal{H}\times H^2$, with the large viscosity assumption
\begin{eqnarray*}
      \mu\geq \mu_0(L,\Gamma,a,b,c,\|u_0\|_{H^1},\|Q_0\|_{H^2}),
\end{eqnarray*} problem (\ref{be})--(\ref{ic}) admits a global strong solution that satisfies
\begin{eqnarray*}
      u\in L^\infty(0,T;\mathcal{H})\cap L^2(0,T;H^2),\ Q\in L^\infty(0,T;H^2)\cap L^2(0,T;H^3).
\end{eqnarray*}
\end{theorem}

Finally, we present a continuous dependence result on the initial data, from which one can infer the weak-strong uniqueness of solutions to problem (\ref{be})--(\ref{ic}).

\begin{theorem}\label{th:1.2}
Suppose that $(u_i,Q_i)$ are global solutions to problem (\ref{be})--(\ref{ic}) corresponding to initial data $(u_{0i},Q_{0i})$, $i=1,2$, respectively. In addition, assume that for any positive $T$, it holds
\begin{eqnarray*}
      &&\|u_1\|_{L^\infty(0,T;L^2)}+\|Q_1\|_{L^\infty(0,T;H^1)}\leq \kappa_1,
      \\
      &&\|u_2\|_{L^\infty(0,T;H^1)}+\|Q_2\|_{L^\infty(0,T;H^2)}\leq \kappa_2,
\end{eqnarray*} for any $t\in[0,T]$. Then we have
\begin{eqnarray*}
      &&\|\delta u\|_{L^2}^2(t)+\|\delta Q\|_{H^1}^2(t)+\int_0^t\left(\mu\|\nabla\delta u\|^2(\tau)+\Gamma L_1^2\|\Delta\delta Q\|^2(\tau)\right){\rm d}s
      \\&\leq& Ce^{Ct}\left(\|\delta u_0\|_{L^2}^2+\|\delta Q_0\|_{H^1}^2\right),
\end{eqnarray*} where $\delta f=f_1-f_2$ and $C$ is a generic constant depending on $\kappa_1$ and $\kappa_2$ but not on $t$.
\end{theorem}

\setcounter{section}{1} \setcounter{equation}{0}
\section{Existence of global weak solutions}
In this section, in order to prove the global existence of weak solutions, we first establish some a priori estimates.

\begin{lemma}
\label{le:1} {\rm(Basic energy equality)} For any $0\le t<T$, there holds
\begin{eqnarray}
    \label{ie:4.0}
    \frac{d}{dt}E(t)+\mu\int_{\mathbb{R}^3}|\nabla u|^2{\rm d}x+\Gamma\int_{\mathbb{R}^3}{\rm tr}(H_Q^2){\rm d}x=0.
\end{eqnarray}
where $E(t)$ is a Lyapunov functional of system (\ref{model}) and defined as follows,
\begin{eqnarray*}
    E(t)\doteq\int_{\mathbb{R}^3}\left(\frac12|u|^2+\frac{L_1}{2}|\nabla Q|^2+\frac{L_2+L_3}{2}|{\rm div}Q|^2+\frac{L_4}{2}e_{l\alpha k}Q_{l\beta}Q_{\alpha\beta,k}+f_{bulk}(Q)\right)(\cdot,t){\rm d}x.
\end{eqnarray*} Furthermore, there holds for any $t\geq 0$,
\begin{eqnarray}
    \label{basic estimate}
    \|u(\cdot,t)\|_{L^2}^2+L_1\|Q(\cdot,t)\|_{H^1}^2+\mu\int_0^t\|\nabla u(\cdot,s)\|_{L^2}^2{\rm d}s+\Gamma L_1^2\int_0^t\|\Delta Q(\cdot,s)\|_{L^2}^2{\rm d}s\leq Ce^{Ct}
\end{eqnarray} with some uniform constants $C$ depending on $a,b,c,\mu,\Gamma,L_i$ and the initial data.
\end{lemma}
\begin{proof} We first multiply (\ref{model})$_1$ by $u$, integrate over $\mathbb{R}^3$ and apply integration by parts, and then sum with (\ref{model})$_2$ multiplied by $H_Q$, taken trace, integrated over $\mathbb{R}^3$ and applied integration by parts. Also note the following fact:
\begin{eqnarray}
    \nonumber
    &&\int_{\mathbb{R}^3}{\rm tr}(Q_t\cdot H_Q){\rm d}x
    \\&=&\nonumber
    -\frac{d}{dt}\int_{\mathbb{R}^3}\left(\frac{L_1}{2}|\nabla Q|^2+\frac{L_2+L_3}{2}|{\rm div}Q|^2+\frac{a}{2}{\rm tr}(Q^2)-\frac{b}{3}{\rm tr}(Q^3)+\frac{c}{4}{\rm tr}^2(Q^2)\right){\rm d}x
    \\&& \nonumber
    \underbrace{+L_4\int_{\mathbb{R}^3}Q_{\alpha\beta,t}e_{l\alpha k}Q_{l\beta,k}{\rm d}x}_{J_1}
    \\&=&
    -\frac{d}{dt}\int_{\mathbb{R}^3}\left(\frac{L_1}{2}|\nabla Q|^2+\frac{L_2+L_3}{2}|{\rm div}Q|^2+\frac{L_4}{2}e_{l\alpha k}Q_{l\beta}Q_{\alpha\beta,k}+f_{bulk}(Q)\right){\rm d}x,
\end{eqnarray}
where
\begin{eqnarray}
    \nonumber
    J_1&=&L_4\frac{d}{dt}\int_{\mathbb{R}^3}Q_{\alpha\beta}e_{l\alpha k}Q_{l\beta,k}{\rm d}x-L_4\int_{\mathbb{R}^3}Q_{\alpha\beta}e_{l\alpha k}Q_{l\beta,kt}{\rm d}x
    \\&=& \nonumber
    L_4\frac{d}{dt}\int_{\mathbb{R}^3}Q_{\alpha\beta}e_{l\alpha k}Q_{l\beta,k}{\rm d}x+L_4\int_{\mathbb{R}^3}Q_{\alpha\beta,k}e_{l\alpha k}Q_{l\beta,t}{\rm d}x
    \\&=&
    L_4\frac{d}{dt}\int_{\mathbb{R}^3}e_{\alpha lk}Q_{l\beta}Q_{\alpha\beta,k}{\rm d}x+L_4\int_{\mathbb{R}^3}Q_{l\beta,k}e_{\alpha lk}Q_{\alpha\beta,t}{\rm d}x,
\end{eqnarray} that is,
\begin{eqnarray}
    J_1=L_4\int_{\mathbb{R}^3}Q_{\alpha\beta,t}e_{l\alpha k}Q_{l\beta,k}{\rm d}x=-\frac{L_4}{2}\frac{d}{dt}\int_{\mathbb{R}^3}e_{l\alpha k}Q_{l \beta}Q_{\alpha\beta,k}{\rm d}x.
\end{eqnarray}
Then we have
\begin{eqnarray*}
    &&\frac{d}{dt}E(t)+\mu\int_{\mathbb{R}^3}|\nabla u|^2{\rm d}x+\Gamma\int_{\mathbb{R}^3}{\rm tr}(H_Q^2){\rm d}x
    \\&=&
    \underbrace{\int_{\mathbb{R}^3}u_\gamma Q_{\alpha\beta,\gamma}\left(-aQ_{\alpha\beta}+b\left[Q_{\alpha\delta} Q_{\delta\beta}-\frac13{\rm tr}(Q^2)\delta_{\alpha\beta}\right]-cQ_{\alpha\beta}{\rm tr}(Q^2)\right){\rm d}x}_{J_2}
    \\&&
    +\underbrace{\int_{\mathbb{R}^3}u_\gamma Q_{\alpha\beta,\gamma}\left[L_1\Delta Q_{\alpha\beta}+\frac12(L_2+L_3)\left(Q_{\alpha \delta,\delta\beta}+Q_{\beta \delta,\delta\alpha}-\frac23\delta_{\alpha\beta}Q_{kp,kp}\right)\right]{\rm d}x}_{J_3}
    \\&&
    \underbrace{+\frac{L_4}{2}\int_{\mathbb{R}^3}u_\gamma Q_{\alpha\beta,\gamma}\left(e_{l\alpha k}Q_{l\beta,k}+e_{l\beta k}Q_{l\alpha,k}-\frac23\delta_{\alpha\beta}e_{lpk}Q_{lp,k}\right){\rm d}x}_{J_4}
    \\&&
    \underbrace{+\frac{L_4}{2}\int_{\mathbb{R}^3}(-\Omega_{\alpha \gamma}Q_{\gamma \beta}+Q_{\alpha \gamma}\Omega_{\gamma \beta})(H_Q)_{\alpha\beta}{\rm d}x}_{J_5}
    \underbrace{-\int_{\mathbb{R}^3}[Q_{\alpha\gamma}(H_Q)_{\gamma\beta}-(H_Q)_{\alpha\gamma}Q_{\gamma\beta}]u_{\alpha,\beta}{\rm d}x}_{J_6}
    \\&&
    \underbrace{+\int_{\mathbb{R}^3}\left(L_1Q_{\gamma\delta,\alpha}Q_{\gamma\delta,\beta} +L_2Q_{\gamma\delta,\delta}Q_{\gamma\beta,\alpha}+L_3Q_{\gamma\beta,\delta}Q_{\gamma\delta,\alpha}
    +\frac{L_4}{2}e_{\gamma\delta\beta}Q_{\gamma m}Q_{\delta m,\alpha}\right)u_{\alpha,\beta}{\rm d}x}_{J_7}.
\end{eqnarray*}

Next, we estimate the terms $J_2, \cdots, J_7$ step by step. First of all, it follows from the incompressibility condition, the symmetry of the $Q$-tensor and integration by parts that
\begin{eqnarray}
    J_2=0.
\end{eqnarray}

On the other hand, using integration by parts and the incompressibility condition, we have
\begin{eqnarray}
    \nonumber
    &&J_3+J_7
    \\&=& \nonumber
    -L_1\int_{\mathbb{R}^3}u_{\gamma,\delta}Q_{\alpha\beta,\gamma}Q_{\alpha\beta,\delta}{\rm d}x-L_2\int_{\mathbb{R}^3}u_{\gamma,\beta} Q_{\alpha\beta,\gamma}Q_{\alpha\delta,\delta}{\rm d}x-L_3\int_{\mathbb{R}^3}u_{\gamma,\delta}Q_{\alpha\beta,\gamma}Q_{\alpha\delta,\beta}{\rm d}x
    \\&& \nonumber
    \underbrace{-L_3\int_{\mathbb{R}^3}u_{\gamma}Q_{\alpha\beta,\gamma\delta}Q_{\alpha\delta,\beta}{\rm d}x}_{J_{8}}+L_1\int_{\mathbb{R}^3}Q_{\gamma\delta,\alpha}Q_{\gamma\delta,\beta}u_{\alpha,\beta}{\rm d}x+L_2\int_{\mathbb{R}^3}Q_{\gamma\delta,\delta}Q_{\gamma\beta,\alpha}u_{\alpha,\beta}{\rm d}x
    \\&&\nonumber
    +L_3\int_{\mathbb{R}^3}Q_{\gamma\beta,\delta}Q_{\gamma\delta,\alpha}u_{\alpha,\beta}{\rm d}x
    +\frac{L_4}{2}\int_{\mathbb{R}^3}e_{\gamma\delta\beta}Q_{\gamma m}Q_{\delta m,\alpha}u_{\alpha,\beta}{\rm d}x
    \\&=& \nonumber
    \frac{L_4}{2}\int_{\mathbb{R}^3}e_{\gamma\delta\beta}Q_{\gamma m}Q_{\delta m,\alpha}u_{\alpha,\beta}{\rm d}x
    \\&=& \nonumber
    -\frac{L_4}{2}\int_{\mathbb{R}^3}e_{\gamma\delta\beta}Q_{\gamma m,\beta}Q_{\delta m,\alpha}u_{\alpha}{\rm d}x-\frac{L_4}{2}\int_{\mathbb{R}^3}e_{\gamma\delta\beta}Q_{\gamma m}Q_{\delta m,\alpha\beta}u_{\alpha}{\rm d}x
    \\&=& \nonumber
    -\frac{L_4}{2}\int_{\mathbb{R}^3}e_{\gamma\delta\beta}Q_{\gamma m,\beta}Q_{\delta m,\alpha}u_{\alpha}{\rm d}x+\frac{L_4}{2}\int_{\mathbb{R}^3}e_{\gamma\delta\beta}Q_{\gamma m,\alpha}Q_{\delta m,\beta}u_{\alpha}{\rm d}x
    \\&=&
    -L_4\int_{\mathbb{R}^3}e_{\gamma\delta\beta}Q_{\gamma m,\beta}Q_{\delta m,\alpha}u_{\alpha}{\rm d}x,
\end{eqnarray} where we have used $J_8=0$. In fact
\begin{eqnarray}
    J_{8}=-L_3\int_{\mathbb{R}^3}u_{\gamma}Q_{\alpha\beta,\gamma\delta}Q_{\alpha\delta,\beta}{\rm d}x=L_3\int_{\mathbb{R}^3}u_{\gamma}Q_{\alpha\beta,\delta}Q_{\alpha\delta,\beta\gamma}{\rm d}x,
\end{eqnarray} which yields $J_{8}=0$.

Then we easily obtain
\begin{eqnarray}
    J_3+J_4+J_7=0.
\end{eqnarray}

It is not difficult to get $J_5+J_6=0$. In fact, this follows from the direct calculations as follows
\begin{eqnarray}
    \nonumber
    J_5&=&\int_{\mathbb{R}^3}\left(\frac{-u_{\alpha,\gamma}+u_{\gamma,\alpha}}{2}Q_{\gamma\beta} +Q_{\alpha\gamma}\frac{u_{\gamma,\beta}-u_{\beta,\gamma}}{2}\right)(H_Q)_{\alpha\beta}{\rm d}x
    \\&=&
    -\int_{\mathbb{R}^3}u_{\alpha,\gamma}Q_{\gamma\beta}(H_Q)_{\alpha\beta}{\rm d}x+\int_{\mathbb{R}^3}u_{\gamma,\alpha}Q_{\gamma\beta}(H_Q)_{\alpha\beta}{\rm d}x=-J_6.
\end{eqnarray}

In conclusion, we have
\begin{eqnarray}
    \label{basic estiamte 2}
    \frac{d}{dt}E(t)+
    \mu\int_{\mathbb{R}^3}|\nabla u|^2{\rm d}x+\Gamma\int_{\mathbb{R}^3}{\rm tr}(H_Q^2){\rm d}x=0.
\end{eqnarray}

In order to obtain the dissipation estimate (\ref{basic estimate}) from (\ref{basic estiamte 2}), we need to estimate the following term contained in $\frac{d}{dt}E(t)$ (see the definition of $E(t)$):
\begin{eqnarray*}
    \frac{d}{dt}\int_{\mathbb{R}^3}\left(\frac{L_4}{2}e_{l\alpha k}Q_{l\beta}Q_{\alpha\beta,k}+f_{bulk}(Q)\right)(\cdot,t){\rm d}x.
\end{eqnarray*}

Multiplying (\ref{model})$_2$ by $Q$, taking the trace, integrating over $\mathbb{R}^3$ and applying integration by parts, we have
\begin{eqnarray}
    \nonumber
    \label{Q^2}
    &&\frac12\frac{d}{dt}\int_{\mathbb{R}^3}|Q|^2{\rm d}x
    \\&=& \nonumber
    \underbrace{\int_{\mathbb{R}^3}{\rm tr}[(\Omega\cdot Q-Q\cdot\Omega)\cdot Q]{\rm d}x}_{J_9}
    -\Gamma L_1\int_{\mathbb{R}^3}|\nabla Q|^2{\rm d}x-\Gamma(L_2+L_3)\int_{\mathbb{R}^3}|{\rm div}Q|^2{\rm d}x
    \\&& \nonumber
    -\Gamma a\int_{\mathbb{R}^3}|Q|^2{\rm d}x
    +\Gamma b\int_{\mathbb{R}^3}{\rm tr}(Q^3){\rm d}x-\Gamma c\int_{\mathbb{R}^3}|Q|^4{\rm d}x
    \\&\leq&
    -\Gamma L_1\int_{\mathbb{R}^3}|\nabla Q|^2{\rm d}x-\Gamma(L_2+L_3)\int_{\mathbb{R}^3}|{\rm div}Q|^2{\rm d}x+C\int_{\mathbb{R}^3}(|Q|^2+|Q|^4){\rm d}x,
\end{eqnarray}
where we have used the fact $J_9=0$.

On the other hand, there exists a sufficiently large constant $K>0$ depending only on $a,b,c$, such that (see \cite{Paicu1,Paicu2})
\begin{eqnarray}
    \frac{K}{2}{\rm tr}(Q^2)+\frac{c}{8}{\rm tr}^2(Q^2)\leq \left(K+\frac{a}{2}\right){\rm tr}(Q^2)-\frac{b}{3}{\rm tr}(Q^3)+\frac{c}{4}{\rm tr}^2(Q^2).
\end{eqnarray}

Meanwhile, it is clear that there exists a constant $C=C(L_1,L_4)>0$ such that
\begin{eqnarray}
    \int_{\mathbb{R}^3}\frac{L_4}{2}e_{l\alpha k}Q_{l\beta}Q_{\alpha\beta,k}{\rm d}x\geq -\frac{L_1}{4}\|\nabla Q\|_{L^2}^2-C\|Q\|_{L^2}^2.
\end{eqnarray}

Multiplying (\ref{Q^2}) by $2(K+C)$, adding the resulting inequality to (\ref{basic estiamte 2}), integrating over $[0,t]$, and then using (2.13) and (2.14), one obtains
\begin{eqnarray}
    \label{gronwall}
    \nonumber
    && \frac12\|u\|_{L^2}^2+\frac{L_1}{4}\|\nabla Q\|_{L^2}^2+\frac{K}{2}\|Q\|_{L^2}^2+\frac{c}{8}\|Q\|_{L^4}^4+\mu\int_0^t\|\nabla u\|_{L^2}^2{\rm d}s+\Gamma\int_0^t\int_{\mathbb{R}^3}{\rm tr}(H_Q^2){\rm d}x
    \\&\leq& \nonumber
    E(t)+(K+C)\|Q\|_{L^2}^2+\mu\int_0^t\|\nabla u\|_{L^2}^2{\rm d}s+\Gamma\int_0^t\int_{\mathbb{R}^3}{\rm tr}(H_Q^2){\rm d}x
    \\&\leq&
    E(0)+(K+C)\|Q_0\|_{L^2}^2
    +C\int_0^t\left(\|\nabla Q\|_{L^2}^2+\|Q\|_{L^2}^2+\|Q\|_{L^4}^4\right){\rm d}s.
\end{eqnarray}

Finally, we only need to estimate the last term on the left hand side of (\ref{gronwall}).  Note that
\begin{eqnarray}
    \nonumber
    \Gamma\int_{\mathbb{R}^3}{\rm tr}(H_Q^2){\rm d}x&=&\Gamma L_1^2\|\Delta Q\|_{L^2}^2+\Gamma\int_{\mathbb{R}^3}{\rm tr}(H_Q-L_1\Delta Q)^2{\rm d}x
    \\&&
    +2\Gamma L_1\int_{\mathbb{R}^3} \Delta Q:(H_Q-L_1\Delta Q){\rm d}x,
\end{eqnarray}
in which
\begin{eqnarray}
    \label{2ab}
    \nonumber
    &&2\Gamma L_1\int_{\mathbb{R}^3} \Delta Q:(H_Q-L_1\Delta Q){\rm d}x
    \\&=&\nonumber
    2\Gamma L_1(L_2+L_3)\|\nabla {\rm div}Q\|_{L^2}^2+\Gamma L_1L_4\int_{\mathbb{R}^3}e_{l\alpha k}Q_{l\beta,k}\Delta Q_{\alpha\beta}{\rm d}x+2a\Gamma L_1\|\nabla Q\|_{L^2}^2
    \\&&\nonumber
    +2b\Gamma L_1\int_{\mathbb{R}^3}\Delta Q_{\alpha\beta}Q_{\alpha\gamma}Q_{\gamma\beta}{\rm d}x-2c\Gamma L_1\int_{\mathbb{R}^3}\Delta Q_{\alpha\beta}Q_{\alpha\beta}{\rm tr}(Q^2){\rm d}x
    \\&=&\nonumber
    2\Gamma L_1(L_2+L_3)\|\nabla {\rm div}Q\|_{L^2}^2+\Gamma L_1L_4\int_{\mathbb{R}^3}e_{l\alpha k}Q_{l\beta,k}\Delta Q_{\alpha\beta}{\rm d}x+2a\Gamma L_1\|\nabla Q\|_{L^2}^2
    \\&&\nonumber
    +2b\Gamma L_1\int_{\mathbb{R}^3}\Delta Q_{\alpha\beta}Q_{\alpha\gamma}Q_{\gamma\beta}{\rm d}x\underbrace{+2c\Gamma L_1\int_{\mathbb{R}^3}|\nabla Q|^2{\rm tr}(Q^2){\rm d}x
    +2c\Gamma L_1\int_{\mathbb{R}^3}|\nabla({\rm tr}(Q^2))|^2{\rm d}x}_{\geq0}
    \\&\geq&
    2\Gamma L_1(L_2+L_3)\|\nabla {\rm div}Q\|_{L^2}^2-\frac12\Gamma L_1^2\|\Delta Q\|_{L^2}^2-C(\|\nabla Q\|_{L^2}^2+\|Q\|_{L^4}^4).
\end{eqnarray}

Hence, (\ref{basic estimate}) follows directly from (\ref{gronwall})-(\ref{2ab}) and the Gronwall inequality.
\end{proof}

Now we turn to construct a global weak solution. The main idea is similar to \cite{Paicu2}. Therefore, we only give a sketch of the proof.

First, we define the mollifying operator
\begin{eqnarray*}
    \widehat{J_nf}(\xi)=1_{[\frac1n,n]}\widehat{f}(\xi),
\end{eqnarray*} where we denote by $\widehat{f}(\xi)$ the Fourier transformation of $f(x)$. Denote by $\mathcal{P}$ the Leray projector into divergence free vector fields. Then we consider the following system:
\begin{eqnarray} \label{mollify}
    \begin{cases}
      u^n_t+\mathcal{P}J_n(\mathcal{P}J_nu^n\nabla\mathcal{P}J_n u^n)=\mu\Delta\mathcal{P}J_nu^n-\mathcal{P}\nabla\cdot J_n[J_nQ^{(n)}\widetilde{H}^{(n)}_Q-\widetilde{H}^{(n)}_QJ_nQ^{(n)}+(\widetilde{\sigma}^d)^{(n)}],
      \\
      Q^{(n)}_t+J_n(\mathcal{P}J_nu^n\nabla J_nQ^{(n)})=\Gamma \widetilde{H}_Q^{(n)}+J_n\left(\mathcal{P}J_n\Omega^n J_nQ^{(n)}\right)-J_n\left(J_nQ^{(n)}\mathcal{P}J_n\Omega^n\right),
\end{cases}
\end{eqnarray}
where
\begin{eqnarray*}
    (\widetilde{H}^{(n)}_Q)_{ij}&\doteq& L_1\Delta J_nQ^{(n)}_{ij}+\frac{L_4}{2}\left(e_{lik}J_nQ^{(n)}_{lj,k}+e_{ljk}J_nQ^{(n)}_{li,k}-\frac23\delta_{ij}e_{lpk}J_nQ_{lp,k}^{(n)}\right)
    \\&&
    +\frac12(L_2+L_3)\left((J_nQ^{(n)})_{ik,kj}+(J_nQ^{(n)})_{jk,ki}-\frac23\delta_{ij}(J_nQ^{(n)})_{kp,kp}\right)
    \\&&
    -aJ_nQ^{(n)}_{ij}+bJ_n\left[J_nQ^{(n)}_{ik}J_nQ^{(n)}_{kj}-\frac{{\rm tr}(J_nQ^{(n)})^2}{3}\delta_{ij}\right]-cJ_n\left(J_nQ^{(n)}_{ij}|J_nQ^{(n)}|^2\right),
\end{eqnarray*} and
\begin{eqnarray*}
    -(\widetilde{\sigma}^d)^{(n)}_{ij}&=&
    L_1J_n(J_nQ^{(n)}_{kl,i}J_nQ^{(n)}_{kl,j})+L_2J_n(J_nQ^{(n)}_{km,m}J_nQ^{(n)}_{kj,i})
    \\&&
    +L_3J_n(J_nQ^{(n)}_{kj,l}J_nQ^{(n)}_{kl,i}) +\frac{L_4}{2}e_{mkj}J_n(J_nQ^{(n)}_{ml}J_nQ_{kl,i}^{(n)}).
\end{eqnarray*}

The above system can be considered as an ordinary differential equation in $L^2$ satisfying the conditions of the Cauchy-Lipschitz theorem. Therefore, it admits a unique local solution $(u^n,Q^{(n)})\in C^1([0,T_n);L^2\times L^2(\mathbb{R}^3;\mathbb{R}^9))$.

\begin{remark} \cite{Paicu1,Paicu2} The operators $J_n$, $\mathcal{P}$ and $\mathcal{P}J_n$ are idempotent. Furthermore, $J_n$ and $\mathcal{P}$ are selfadjoint in $L^2$. $J_n$ commute with distributional derivatives.
\end{remark}

From this remark, we conclude that the pair $(\mathcal{P}J_nu^n,J_nQ^{(n)})$ is also a solution of system (\ref{mollify}). Then we know that $(u^n,Q^{(n)})=(\mathcal{P}J_nu^n,J_nQ^{(n)})\in C^1([0,T_n),H^\infty)$ solves the following system:
\begin{eqnarray} \label{mollify2}
    \begin{cases}
      u^n_t+\mathcal{P}J_n(u^n\nabla u^n)=\mu\Delta u^n-\mathcal{P}\nabla\cdot J_n[Q^{(n)}H^{(n)}_Q-H^{(n)}_QQ^{(n)}+(\sigma^d)^{(n)}],
      \\
      Q^{(n)}_t+J_n(u^n\nabla Q^{(n)})=\Gamma H_Q^{(n)}+J_n\left(\Omega^n Q^{(n)}\right)-J_n\left(Q^{(n)}\Omega^n\right),
\end{cases}
\end{eqnarray}
in which
\begin{eqnarray*}
    (H^{(n)}_Q)_{ij}&\doteq& L_1\Delta Q^{(n)}_{ij}+\frac{L_4}{2}\left(e_{lik}Q^{(n)}_{lj,k}+e_{ljk}Q^{(n)}_{li,k}-\frac23\delta_{ij}e_{lpk}Q_{lp,k}^{(n)}\right)
    \\&&
    +\frac12(L_1+L_2)\left((Q^{(n)})_{ik,kj}+(Q^{(n)})_{jk,ki}-\frac23\delta_{ij}(Q^{(n)})_{kp,kp}\right)
    \\&&
    -aQ^{(n)}_{ij}+bJ_n\left[Q^{(n)}_{ik}Q^{(n)}_{kj}-\frac{{\rm tr}(Q^{(n)})^2}{3}\delta_{ij}\right]-cJ_n\left(Q^{(n)}_{ij}|Q^{(n)}|^2\right),
\end{eqnarray*} and
\begin{eqnarray*}
    -(\sigma^d)^{(n)}_{ij}&=&
    L_1J_n(Q^{(n)}_{kl,i}Q^{(n)}_{kl,j})+L_2J_n(Q^{(n)}_{km,m}Q^{(n)}_{kj,i})
    \\&&
    +L_3J_n(Q^{(n)}_{kj,l}Q^{(n)}_{kl,i})+\frac{L_4}{2}e_{mkj}J_n(Q^{(n)}_{ml}Q^{(n)}_{kl,j}).
\end{eqnarray*}

The estimates for the sequence $\{(u^n,Q^{(n)})\}$ are exactly the same as above. Therefore, we get that for any given positive $T$,
\begin{eqnarray*}
    &&\sup_n \|u^n\|_{L^\infty(0,T;L^2)\cap L^2(0,T;H^1)}<\infty,\\
    &&\sup_n \|Q^{(n)}\|_{L^\infty(0,T;H^1)\cap L^2(0,T;H^2)}<\infty.
\end{eqnarray*}

On the other hand, we can get the bounds on $(\partial_t u^n,\partial_t Q^{(n)})$ in some $L^\infty_{loc}(H^{-k})$ for some large $k$ from the equations. Then by Aubin-Lions compactness lemma, after taking possible subsequences, we may obtain
\begin{eqnarray*}
    &&u^n\rightharpoonup u\ {\rm weakly*}\ {\rm in}\ L^\infty(0,T;L^2),\  u^n\rightharpoonup u\ {\rm weakly}\ {\rm in}\ L^2(0,T;H^1),
    \\&&
    u^n\rightarrow u\ {\rm strongly}\ {\rm in}\ L^2(0,T;H^{1-\varepsilon}_{loc}),\ \forall \varepsilon>0,
    \\&&
    u^n(t)\rightharpoonup u(t)\ {\rm weakly}\ {\rm in}\ L^2\ {\rm for}\ {\rm all}\ t>0;
    \\&&
    Q^{(n)}\rightharpoonup Q\ {\rm weakly*}\ {\rm in}\ L^\infty(0,T;H^1),\ Q^{(n)}\rightharpoonup Q\ {\rm weakly}\ {\rm in}\ L^2(0,T;H^2),
    \\&&
    Q^{(n)}\rightarrow Q\ {\rm strongly}\ {\rm in}\ L^2(0,T;H^{2-\varepsilon}_{loc}),\ \forall \varepsilon>0,
    \\&&
    Q^{(n)}(t)\rightharpoonup Q(t)\ {\rm weakly}\ {\rm in}\ H^1\ {\rm for}\ {\rm all}\ t>0,
\end{eqnarray*}
which is enough for us to pass to the limit in the weak solutions of (\ref{mollify2}) and Theorem \ref{weak solutions} follows from some diagonal arguments. \qed

\setcounter{section}{2} \setcounter{equation}{0}
\section{Global existence of strong solutions}
In this section, we intend to establish the higher regularity of the global weak solutions with sufficiently large viscosity of the fluid by using the energy argument shown in \cite{Sun-Liu,Wu-Xu-Liu1,Wu-Xu-Liu2}. In order to finish the energy estimates, we need the following well-known lemma for Cauchy problem in dimension three.
\begin{lemma}\label{le:3.3}(Gagliardo-Nirenberg inequality) For $p\in [2,6]$, $q\in (1,\infty)$, and $r\in (3,\infty)$, there exists some generic constant $C>0$ which may depend on $q$, $r$ such that for $f\in H^1$ and $g\in L^q\cap W^{1,r}$, we have
\begin{eqnarray}
    \|f\|_{L^p}^p\leq C\|f\|_{L^2}^{\frac{6-p}{2}}\|\nabla f\|_{L^2}^{\frac{3p-6}{2}},
\end{eqnarray}
and
\begin{eqnarray}
    \|g\|_{C(\mathbb{R}^3)}\leq C\|g\|_{L^q}^{\frac{q(r-3)}{3r+q(r-3)}}\|\nabla g\|_{L^r}^{\frac{3r}{3r+q(r-3)}}.
\end{eqnarray}
\end{lemma}

Let
$$\mathcal{A}(t)=\|\nabla u\|_{L^2}^2(t)+L_1\|\Delta Q\|_{L^2}^2(t)+(L_2+L_3)\|\nabla{\rm div}Q\|_{L^2}^2(t),$$
and $\widetilde{\mathcal{A}}(t)=\mathcal{A}(t)+1$. Then we have the following lemma, which will directly yield the higher order estimates with sufficiently large viscosity.
\begin{lemma}
\label{le:2} For any given $T>0$, there holds in $[0, T)$ the following estimates
\begin{eqnarray}
    \label{est2}
    \frac{d}{dt}\widetilde{\mathcal{A}}(t)+\left(\mu-C_1\mu^{\frac12}\widetilde{\mathcal{A}}(t)\right)\|\Delta u\|_{L^2}^2
    +\left(\frac{\Gamma L_1^2}{2}-C_2\mu^{-\frac14}\widetilde{\mathcal{A}}(t)\right)\|\nabla\Delta Q\|_{L^2}^2\leq C_3\widetilde{\mathcal{A}}(t),
\end{eqnarray} where $C_i$ $(i=1,2,3)$ are constants depending on $\|u_0\|_{L^2}$, $\|Q_0\|_{H^1}$ and $\mu_0$ which is the positive lower bound of $\mu$.
\end{lemma}
\begin{proof} It follows from direct calculation and integration by parts that
\begin{eqnarray}
    \nonumber
    &&\frac12\frac{d}{dt}\widetilde{\mathcal{A}}(t)
    \\&=& \nonumber
    -\langle \Delta u,\partial_t u\rangle +L_1\langle \Delta Q_{\alpha\beta},\partial_t\Delta Q_{\alpha\beta}\rangle +(L_2+L_3)\langle Q_{\alpha\delta,\delta\gamma},\partial_tQ_{\alpha\beta,\beta\gamma}\rangle
    \\&=& \nonumber
    -\langle \Delta u,\partial_t u\rangle +\langle L_1\Delta Q_{\alpha\beta}+(L_2+L_3)Q_{\alpha\delta,\delta\beta},\partial_t\Delta Q_{\alpha\beta}\rangle
    \\&=& \nonumber
    -\langle \Delta u,\partial_t u\rangle +\langle M_Q,\partial_t\Delta Q\rangle
    \\&=&\nonumber
    \underbrace{\langle \Delta u_\alpha,u_\gamma\partial_\gamma u_\alpha\rangle }_{K_1}-\mu\|\Delta u\|_{L^2}^2
    \\&& \nonumber
    \underbrace{+\langle \Delta u_\alpha,\partial_\beta(L_1Q_{\gamma\delta,\alpha}Q_{\gamma\delta,\beta}+L_2Q_{\gamma\delta,\delta}Q_{\gamma\beta,\alpha} +L_3Q_{\gamma\beta,\delta}Q_{\gamma\delta,\alpha}+\frac{L_4}{2}e_{m \gamma\beta}Q_{m\delta}Q_{\gamma\delta,\alpha})\rangle }_{K_2}
    \\&&\nonumber
    \underbrace{-\langle \Delta u_\alpha,\partial_\beta(Q_{\alpha\gamma}(M_Q)_{\gamma\beta}-(M_Q)_{\alpha\gamma}Q_{\gamma\beta})\rangle }_{K_3}
    \\&& \nonumber
    \underbrace{-\langle \Delta u_\alpha,\partial_\beta(Q_{\alpha\gamma}(E_Q)_{\gamma\beta}
    -(E_Q)_{\alpha\gamma}Q_{\gamma\beta})\rangle }_{K_4}
    \\&& \nonumber
    \underbrace{-\langle L_1\Delta Q_{\alpha\beta}+(L_2+L_3)Q_{\alpha\delta,\delta\beta},\Delta(u_\gamma\partial_\gamma Q_{\alpha\beta})\rangle }_{K_5}
    \\&&\nonumber
    +\underbrace{\langle (M_Q)_{\alpha\beta},\Delta(\Omega_{\alpha\gamma} Q_{\gamma\beta})-\Delta(Q_{\alpha\gamma}\Omega_{\gamma\beta})\rangle }_{K_6}\underbrace{+
    \Gamma\langle (M_Q)_{\alpha\beta},\Delta(M_Q)_{\alpha\beta}\rangle }_{K_7}
    \\&&
    \underbrace{+\Gamma \langle (M_Q)_{\alpha\beta},\Delta(E_Q)_{\alpha\beta}\rangle }_{K_8}
    \underbrace{+\Gamma \langle (M_Q)_{\alpha\beta},\Delta((B_Q)_{\alpha\beta})\rangle }_{K_{9}},
\end{eqnarray} where we have used $Q_{\alpha\gamma}(B_Q)_{\gamma\beta}-(B_Q)_{\alpha\gamma}Q_{\gamma\beta}=0$.

Note first that
\begin{eqnarray}
    \nonumber K_7&=& -\Gamma\|\nabla M_Q\|_{L^2}^2
    \\&=& \nonumber
    -\Gamma L_1^2\|\nabla\Delta Q\|_{L^2}^2-\Gamma L_1(L_2+L_3)\langle \nabla\Delta Q_{\alpha\beta},\nabla(Q_{\alpha\gamma,\gamma\beta}+Q_{\beta\gamma,\gamma\alpha}-\frac23\delta_{\alpha\beta}Q_{\gamma\delta,\gamma\delta})\rangle
    \\&& \nonumber
    -\Gamma\left(\frac{L_2+L_3}{2}\right)^2\|\nabla(Q_{\alpha\gamma,\gamma\beta}+Q_{\beta\gamma,\gamma\alpha}-\frac23\delta_{\alpha\beta}Q_{\gamma\delta,\gamma\delta})\|_{L^2}^2
    \\&\leq&
    -\Gamma L_1^2\|\nabla\Delta Q\|_{L^2}^2-2\Gamma L_1(L_2+L_3)\|\Delta{\rm div}Q\|_{L^2}^2.
\end{eqnarray}

Secondly, we need the following lemma.

\begin{lemma} (Key estimates) There holds
\begin{eqnarray}
    K_3+K_6
    \leq
    C_1\mu^{\frac12}\widetilde{\mathcal{A}}(t)\|\Delta u\|_{L^2}^2+C_2\mu^{-\frac14}\widetilde{\mathcal{A}}(t)\|\nabla\Delta Q\|_{L^2}^2
    +C_3(\mu^{-\frac54}+1)\|\Delta Q\|_{L^2}^2.
\end{eqnarray}
\end{lemma}
\begin{proof} From the symmetry and direct calculations, we have
\begin{eqnarray}
    \nonumber
    K_6&=& \nonumber
    \frac12\langle (M_Q)_{\alpha\beta},\Delta(u_{\alpha,\gamma}Q_{\gamma\beta})\rangle -\frac12\langle (M_Q)_{\alpha\beta},\Delta(u_{\gamma,\alpha}Q_{\gamma\beta})\rangle
    \\&&\nonumber
    -\frac12\langle (M_Q)_{\alpha\beta},\Delta(Q_{\alpha\gamma}u_{\gamma,\beta})\rangle +\frac{1}{2}\langle (M_Q)_{\alpha\beta},\Delta(Q_{\alpha\gamma }u_{\beta,\gamma})\rangle
    \\&=& \nonumber
    \langle (M_Q)_{\alpha\beta},\Delta(u_{\alpha,\gamma}Q_{\gamma\beta})\rangle -\langle (M_Q)_{\alpha\beta},\Delta(u_{\gamma,\alpha}Q_{\gamma\beta})\rangle
    \\&=&\nonumber
    \underbrace{\langle (M_Q)_{\alpha\beta},\Delta u_{\alpha,\gamma}Q_{\gamma\beta}\rangle }_{K_{6,1}}+\langle (M_Q)_{\alpha\beta},u_{\alpha,\gamma}\Delta Q_{\gamma\beta}\rangle
    +2\langle (M_Q)_{\alpha\beta}, u_{\alpha,\gamma\delta} Q_{\gamma\beta,\delta}\rangle
    \\&&\nonumber
    \underbrace{-\langle (M_Q)_{\alpha\beta},\Delta u_{\gamma,\alpha}Q_{\gamma\beta}\rangle }_{K_{6,2}}-\langle (M_Q)_{\alpha\beta},u_{\gamma,\alpha}\Delta Q_{\gamma\beta}\rangle
    -2\langle (M_Q)_{\alpha\beta},u_{\gamma,\alpha \delta}Q_{\gamma\beta,\delta}\rangle .
\end{eqnarray}
which yields that
\begin{eqnarray}
    K_3+K_{6,1}+K_{6,2}=0.
\end{eqnarray}

Then by the H\"older inequality, the Sobolev inequalities, Lemma \ref{le:1}, Lemma \ref{le:3.3} and the Cauchy inequality, we have
\begin{eqnarray}
    \nonumber
    &&K_3+K_6
    \\&\leq& \nonumber
    C\int_{\mathbb{R}^3}\left(|\nabla u||\nabla^2 Q|^2+|\nabla^2u||\nabla Q||\nabla^2 Q|\right){\rm d}x.
    \\&\leq&\nonumber
    C\|\nabla u\|_{L^6}\|\Delta Q\|_{L^2}\|\Delta Q\|_{L^3}+C\|\nabla^2u\|_{L^2}\|\Delta Q\|_{L^2}\|\nabla Q\|_{L^\infty}
    \\&\leq&\nonumber
    C\|\Delta u\|_{L^2}\|\Delta Q\|_{L^2}^{\frac32}\|\nabla\Delta Q\|_{L^2}^{\frac12}+C\|\nabla^2u\|_{L^2}\|\Delta Q\|_{L^2}\|\nabla\Delta Q\|_{L^2}^{\frac34}
    \\&\leq&\nonumber
    \mu^{\frac12}\|\Delta Q\|_{L^2}^2\|\Delta u\|_{L^2}^2+C\mu^{-\frac14}\|\nabla\Delta Q\|_{L^2}^2+C\mu^{-\frac34}\|\Delta Q\|_{L^2}^2
    \\&&\nonumber
    +\mu^{\frac12}\|\Delta u\|_{L^2}^2+C\mu^{-\frac14}\|\Delta Q\|_{L^2}^2\|\nabla\Delta Q\|_{L^2}^2+C\mu^{-\frac54}\|\Delta Q\|_{L^2}^2
    \\&\leq& \nonumber
    \mu^{\frac12}\widetilde{\mathcal{A}}(t)\|\Delta u\|_{L^2}^2+C\mu^{-\frac14}\widetilde{\mathcal{A}}(t)\|\nabla\Delta Q\|_{L^2}^2
    +C(\mu^{-\frac54}+\mu^{-\frac34})\|\Delta Q\|_{L^2}^2. \qed
\end{eqnarray}

Now we present the estimates about the rest terms. It follows from integration by parts that
\begin{eqnarray}
    \nonumber
    K_2&=&L_1\langle \Delta u_\alpha,Q_{\gamma\delta,\alpha}\Delta Q_{\gamma\delta}\rangle +L_2\langle \Delta u_\alpha,Q_{\gamma\delta,\delta\beta}Q_{\gamma\beta,\alpha}\rangle
    \\&& \nonumber
    +L_3\langle \Delta u_\alpha,Q_{\gamma\beta,\beta\delta}Q_{\gamma\delta,\alpha}+Q_{\gamma\beta,\delta}Q_{\gamma\delta,\alpha\beta}\rangle
    \\&& \nonumber
    +\frac{L_4}{2}\langle \Delta u_\alpha,e_{m\gamma\beta}Q_{m\delta,\beta}Q_{\gamma\delta,\alpha}\rangle +\frac{L_4}{2}\langle \Delta u_\alpha,e_{m\gamma\beta}Q_{m\delta}Q_{\gamma\delta,\alpha\beta}\rangle
    \\&=& \nonumber
    \underbrace{L_1\langle \Delta u_\alpha,Q_{\gamma\delta,\alpha}\Delta Q_{\gamma\delta}\rangle }_{K_{2,1}}+\underbrace{(L_2+L_3)\langle \Delta u_\alpha,Q_{\gamma\delta,\delta\beta}Q_{\gamma\beta,\alpha}\rangle }_{K_{2,2}}
    \\&& \nonumber
    +L_3\langle \Delta u_\alpha,Q_{\gamma\beta,\delta}Q_{\gamma\delta,\alpha\beta}\rangle +\frac{L_4}{2}\langle \Delta u_\alpha,e_{m\gamma\beta}Q_{m\delta,\beta}Q_{\gamma\delta,\alpha}\rangle
    \\&&
    +\frac{L_4}{2}\langle \Delta u_\alpha,e_{m\gamma\beta}Q_{m\delta}Q_{\gamma\delta,\alpha\beta}\rangle ,
\end{eqnarray} and
\begin{eqnarray}
    \nonumber
    K_5&=& \underbrace{-L_1\langle \Delta Q_{\alpha\beta},\Delta u_\gamma Q_{\alpha\beta,\gamma}\rangle }_{K_{5,1}}-2L_1\langle \Delta Q_{\alpha\beta},u_{\gamma,\delta}Q_{\alpha\beta,\gamma\delta}\rangle
    \\&& \nonumber
    \underbrace{-(L_2+L_3)\langle Q_{\alpha\delta,\delta\beta},\Delta u_\gamma Q_{\alpha\beta,\gamma}\rangle }_{K_{5,2}}-2(L_2+L_3)\langle Q_{\alpha\delta,\delta\beta},u_{\gamma,k}Q_{\alpha\beta,\gamma k}\rangle
    \\&&
    \underbrace{-(L_2+L_3)\langle Q_{\alpha\delta,\delta\beta},u_\gamma\Delta Q_{\alpha\beta,\gamma}\rangle }_{K_{5,3}}.
\end{eqnarray}
Note that $K_{2,1}+K_{5,1}=K_{2,2}+K_{5,2}=0$. Integrating by parts twice, one obtains
\begin{eqnarray}
    K_{5,3}=-(L_2+L_3)\langle Q_{\alpha\delta,\delta k},u_{\gamma,\beta}Q_{\alpha\beta,\gamma k}\rangle +(L_2+L_3)\langle Q_{\alpha\delta,\delta\beta},u_{\gamma,k}Q_{\alpha\beta,\gamma k}\rangle .
\end{eqnarray} Then, similarly as above, we get
\begin{eqnarray}
    \nonumber
    &&K_2+K_5
    \\&\leq& \nonumber
    C\int_{\mathbb{R}^3}\left(|\Delta u||\nabla Q||\nabla^2 Q|+|\Delta u||\nabla Q|^2+|\Delta u||Q||\nabla^2Q|+|\nabla u||\nabla^2 Q|^2\right){\rm d}x
    \\&\leq& \nonumber
    C\|\nabla Q\|_{L^\infty}\|\Delta u\|_{L^2}\|\Delta Q\|_{L^2}+C\|\Delta u\|_{L^2}\|\nabla Q\|_{L^6}\|\nabla Q\|_{L^3}
    \\&& \nonumber
    +C\|\Delta u\|_{L^2}\|Q\|_{L^6}\|\nabla^2Q\|_{L^3}+C\|\nabla u\|_{L^6}\|\Delta Q\|_{L^2}\|\Delta Q\|_{L^3}
    \\&\leq& \nonumber
    C\|\Delta u\|_{L^2}\|\Delta Q\|_{L^2}\|\nabla\Delta Q\|_{L^2}^{\frac34}+C\|\Delta u\|_{L^2}\|\Delta Q\|_{L^2}^{\frac32}
    \\&& \nonumber
    +C\|\Delta u\|_{L^2}\|\Delta Q\|_{L^2}^{\frac12}\|\nabla\Delta Q\|_{L^2}^{\frac12}+C\|\Delta u\|_{L^2}\|\Delta Q\|_{L^2}^{\frac32}\|\nabla\Delta Q\|_{L^2}^{\frac12}
    \\&\leq&
    \mu^{\frac12}\widetilde{\mathcal{A}}(t)\|\Delta u\|_{L^2}^2+C\mu^{-\frac14}\widetilde{\mathcal{A}}(t)\|\nabla\Delta Q\|_{L^2}^2
    +C(\mu^{-\frac54}+\mu^{-\frac34}+\mu^{-\frac12})\|\Delta Q\|_{L^2}^2,
\end{eqnarray}

Meanwhile, we also have
\begin{eqnarray}
    \nonumber
    K_1&\leq& \|u\|_{L^4}\|\nabla u\|_{L^4}\|\Delta u\|_{L^2}\leq
    C\|u\|_{L^2}^{\frac14}\|\nabla u\|_{L^2}^{\frac34}\|\nabla u\|_{L^2}^{\frac14}\|\Delta u\|_{L^2}^{\frac34}\|\Delta u\|_{L^2}
    \\&\leq&\nonumber
    \mu^{\frac12}\|\Delta u\|_{L^2}^2+\mu^{\frac12}\|\nabla u\|_{L^2}^2\|\Delta u\|_{L^2}^2+C\mu^{-\frac72}\|\nabla u\|^2
    \\&\leq&
    \mu^{\frac12}\widetilde{\mathcal{A}}(t)\|\Delta u\|_{L^2}^2+C\mu^{-\frac72}\widetilde{\mathcal{A}}(t),
\end{eqnarray} and
\begin{eqnarray}
    \nonumber
    &&K_4+K_8
    \\&\leq& \nonumber
    C\int_{\mathbb{R}^3}\left(|\Delta u||\nabla Q|^2+|\Delta u||Q||\nabla^2 Q|+|\nabla^2 Q||\nabla\Delta Q|\right){\rm d}x
    \\&\leq& \nonumber
    C\|\Delta u\|_{L^2}\|\nabla Q\|_{L^6}\|\nabla Q\|_{L^3}+C\|\Delta u\|_{L^2}\|Q\|_{L^6}\|\nabla^2Q\|_{L^3}
    +C\|\Delta Q\|_{L^2}\|\nabla\Delta Q\|_{L^2}
    \\&\leq&
    \mu^{\frac12}\widetilde{\mathcal{A}}(t)\|\Delta u\|_{L^2}^2+\left(\frac{\Gamma L_1^2}{4}+C\mu^{-\frac14}\widetilde{\mathcal{A}}(t)\right)\|\nabla\Delta Q\|_{L^2}^2+C(\mu^{-\frac34}+\mu^{-\frac12}+1)\widetilde{\mathcal{A}}(t).\quad\quad
\end{eqnarray}

Finally,
\begin{eqnarray}
    \nonumber
    K_{9}&\leq& \nonumber
    -a\Gamma L_1\|\Delta Q\|_{L^2}^2-a\Gamma(L_2+L_3)\|\nabla{\rm div} Q\|_{L^2}^2
    \\&& \nonumber
    +C\int_{\mathbb{R}^3}\left[|\nabla^2 Q|^2\left(|Q|^2+|Q|\right)+|\nabla^2 Q||\nabla Q|^2(|Q|+1)\right]{\rm d}x
    \\&\leq& \nonumber
    -a\Gamma L_1\|\Delta Q\|_{L^2}^2-a\Gamma(L_2+L_3)\|\nabla{\rm div} Q\|_{L^2}^2
    \\&& \nonumber
    +C\|\nabla^2 Q\|_{L^6}\|\nabla^2 Q\|_{L^2}\left(\|Q\|_{L^6}^2+\|Q\|_{L^3}\right)
    \\&& \nonumber
    +\|\nabla^2 Q\|_{L^6}\|\nabla Q\|_{L^2}\|\nabla Q\|_{L^6}\|Q\|_{L^6}+\|\nabla^2 Q\|_{L^6}\|\nabla Q\|_{L^3}\|\nabla Q\|_{L^2}
    \\&\leq&
    \frac{\Gamma L_1^2}{4}\|\nabla\Delta Q\|_{L^2}^2+C(\|\Delta Q\|_{L^2}^2+1).
\end{eqnarray}

In conclusion, we finish the proof of (\ref{est2}).
\end{proof}

From the proof of (2.15)-(2.17) and the Gronwall inequality, one has
\begin{eqnarray}
    \int_0^t\|\nabla u\|_{L^2}^2{\rm d}s+\int_0^t \left[L_1\|\Delta Q\|_{L^2}^2+(L_2+L_3)\|\nabla{\rm div}Q\|^2\right]{\rm d}s\leq Ce^{Ct}
\end{eqnarray} with some uniform constants $C$ depending on $a,b,c,\mu,\Gamma,L_i$ and the initial data. It yields that
$\widetilde{\mathcal{A}}(t)$ is integrable over $[0,T]$ for any positive $T$ (assume that $T>1$). Then there exists $M>0$ depending on $C$ and $T$ such that
\begin{eqnarray}
    \int_t^{t+1}\widetilde{\mathcal{A}}(t){\rm d}s\leq M
\end{eqnarray}
for any $t\in [0,T-1]$.

If
\begin{eqnarray}
    \mu^{\frac12}\geq \mu_0^{\frac12}\doteq C_1(\widetilde{\mathcal{A}}(0)+C_3M+2M)+\frac{4C_2^2}{\Gamma^2L_1^2}(\widetilde{\mathcal{A}}(0)+C_3M+2M)^2+1,
\end{eqnarray}
then initially there is some $T_0>0$ such that
\begin{eqnarray}
    \label{new}
    \mu-C_1\mu^{\frac12}\widetilde{\mathcal{A}}(t)\geq 0,\
    \frac{\Gamma L_1}{2}-C_2\mu^{-\frac14}\widetilde{\mathcal{A}}(t)\geq 0
\end{eqnarray} for all $t\in[0,T_0]$. Therefore, in this time interval we have
\begin{eqnarray}
    \frac{d}{dt}\widetilde{\mathcal{A}}(t)\leq C_3\widetilde{\mathcal{A}}(t).
\end{eqnarray}

We assume that $T_*$ is the largest one among such $T_0$ and claim that $T_*=T$. We argue by contradiction and only consider the breakdown of first inequality of (\ref{new}). Suppose that the first equations of (\ref{new}) is not always valid when $T_*<t\leq T$, and $\mu-C_1\mu^{\frac12}\widetilde{\mathcal{A}}(T_*)= 0$. In fact,
if $T_*\leq 1$, then
\begin{eqnarray*}
    C_1(\widetilde{\mathcal{A}}(0)+C_3M+2M)+1\leq  \mu^{\frac12}=C_1\widetilde{\mathcal{A}}(T_*)
    &\leq& C_1\left[\widetilde{\mathcal{A}}(0)+\int_0^{T_*}\frac{d}{dt}\widetilde{\mathcal{A}}(t){\rm d}t\right]
    \\ &\leq&
    C_1\left[\widetilde{\mathcal{A}}(0)+\int_0^{T_*}C_3\widetilde{\mathcal{A}}(t){\rm d}t\right]
    \\&\leq&
    C_1\left[\widetilde{\mathcal{A}}(0)+C_3M\right],
\end{eqnarray*}
which yields a contradiction.

If $1<T_*<T$, then we consider the interval $[T_*-1,T_*]$. From the definition of $M$, there exists $t_*\in (T_*-1,T_*)$
such that $\widetilde{\mathcal{A}}(t_*)\leq 2M$. Then
\begin{eqnarray*}
    C_1(\widetilde{\mathcal{A}}(0)+C_3M+2M)+1\leq  \mu^{\frac12}=C_1\widetilde{\mathcal{A}}(T_*)
    &\leq& C_1\left[\widetilde{\mathcal{A}}(t_*)+\int_{t_*}^{T_*}\frac{d}{dt}\widetilde{\mathcal{A}}(t){\rm d}t\right]
    \\ &\leq&
    2C_1M+C_1C_3M,
\end{eqnarray*}
which also yields a contradiction.

In conclusion, Lemma \ref{le:2} and the discussions above show that the solution $(u,Q)$ satisfies the following regularity
\begin{eqnarray*}
      u\in L^\infty(0,T;\mathcal{H})\cap L^2(0,T;H^2),\ Q\in L^\infty(0,T;H^2)\cap L^2(0,T;H^3),
\end{eqnarray*}
which completes the proof of Theorem \ref{th:1.1}.
\end{proof}
\setcounter{section}{3} \setcounter{equation}{0}
\section{Continuous dependence on initial data}

\noindent {\bf Proof of Theorem 2.2}: We first get the following system with respect to $(\delta u,\delta Q)$ from (\ref{model}):
\begin{eqnarray}
\begin{cases}\label{delta}
    \partial_t\delta u+\mathcal{P}(\delta u\cdot\nabla\delta u)=\mu\Delta\delta u-\mathcal{P}(\nabla\cdot(\frac{\partial \mathcal{F}_{LG}[\delta Q]}{\partial \nabla\delta Q}\odot\nabla\delta Q))
    \\
    \quad
    +\mathcal{P}(\nabla\cdot(\delta Q (M_{\delta Q}+E_{\delta Q})-(M_{\delta Q}+E_{\delta Q})\delta Q))-\mathcal{P}(u_2\cdot\nabla\delta u+\delta u\cdot\nabla u_2)
    \\
    \quad
    +\mathcal{P}(\nabla\cdot(\delta Q (M_{Q_2}+E_{Q_2})-(M_{Q_2}+E_{Q_2})\delta Q))
    \\ \quad
    +\mathcal{P}(\nabla\cdot(Q_2(M_{\delta Q}+E_{\delta Q})-(M_{\delta Q}+E_{\delta Q})Q_2))
    \\
    \quad
    -\mathcal{P}(\nabla\cdot(\frac{\partial \mathcal{F}_{LG}[\delta Q]}{\partial \nabla\delta Q}\odot\nabla Q_2))-\mathcal{P}(\nabla\cdot(\frac{\partial \mathcal{F}_{LG}[Q_2]}{\partial \nabla\delta Q_2}\odot\nabla\delta Q)),
    \\
    (\partial_t+\delta u\cdot\nabla)\delta Q-\delta\Omega\delta Q+\delta Q\delta\Omega+\delta u\nabla Q_2+u_2\cdot\nabla\delta Q \\
    \quad +Q_2\delta\Omega-\delta\Omega Q_2+\delta Q\Omega_2-\Omega_2\delta Q
    \\
    =\Gamma\left(M_{\delta Q}+E_{\delta Q}-a\delta Q+b[\delta QQ_1+Q_2\delta Q-\frac{{\rm tr}(\delta QQ_1+Q_2\delta Q)}{3}I_3]\right)
    \\ \quad
    -\Gamma c(\delta Q{\rm tr}(Q_1^2)+Q_2[{\rm tr}(Q_1\delta Q+\delta QQ_2)]).
\end{cases}
\end{eqnarray}

Multiplying the second equation of system (\ref{delta}) by $(M_{\delta Q}+E_{\delta Q}+K\delta Q)$, taking the trace and integrating over $\mathbb{R}^3$ and then summing with the first equation multiplied by $\delta u$ and integrated over $\mathbb{R}^3$, we have
\begin{eqnarray}
    \nonumber
    &&\frac12\frac{d}{dt}\int_{\mathbb{R}^3}\left(K|\delta Q|^2+L_1|\nabla\delta Q|^2+(L_2+L_3)|{\rm div}\delta Q|^2+L_4e_{l\alpha k}\delta Q_{l\beta}\delta Q_{\alpha\beta,k}\right){\rm d}x
    \\&&\nonumber
    +\mu\|\nabla\delta u\|_{L^2}^2+ \Gamma\|M_{\delta Q}+E_{\delta Q}\|_{L^2}^2
    \\&=&\nonumber
    \underbrace{\int_{\mathbb{R}^3}{\rm tr}\left([\delta u\nabla Q_2+u_2\nabla\delta Q+\delta Q\Omega_2-\Omega_2\delta Q](M_{\delta Q}+E_{\delta Q})\right){\rm d}x}_{N_1}
    \\&&\nonumber
    \underbrace{-a\Gamma L_1\|\nabla\delta Q\|_{L^2}^2-a\Gamma (L_2+L_3)\|{\rm div}\delta Q\|_{L^2}^2}_{N_2}
    \underbrace{+a\Gamma\int_{\mathbb{R}^3}{\rm tr}(\delta QE_{\delta Q}){\rm d}x}_{N_3}
    \\&&\nonumber
    \underbrace{-b\Gamma \int_{\mathbb{R}^3}{\rm tr}((\delta QQ_1+Q_2\delta Q)(M_{\delta Q}+E_{\delta Q})){\rm d}x}_{N_4}
    \underbrace{+c\Gamma \int_{\mathbb{R}^3}{\rm tr}(\delta Q(M_{\delta Q}+E_{\delta Q})){\rm tr}(Q_1^2){\rm d}x}_{N_5}
    \\&&\nonumber
    \underbrace{+c\Gamma \int_{\mathbb{R}^3}{\rm tr}(Q_2(M_{\delta Q}+E_{\delta Q})){\rm tr}(Q_1\delta Q+\delta QQ_2){\rm d}x}_{N_6}
    \underbrace{-K\int_{\mathbb{R}^3}{\rm tr}(\delta u\nabla Q_2\delta Q){\rm d}x}_{N_7}
    \\&&\nonumber
    \underbrace{-K\int_{\mathbb{R}^3}{\rm tr}(Q_2\delta\Omega\delta Q){\rm d}x+K\int_{\mathbb{R}^3}{\rm tr}(\delta\Omega Q_2\delta Q){\rm d}x}_{N_8}
    \underbrace{-K\Gamma L_1\|\nabla Q\|_{L^2}^2}_{N_9} \underbrace{-K\Gamma(L_2+L_3)\|{\rm div} Q\|_{L^2}^2}_{N_{10}}
    \\&&\nonumber
    \underbrace{+K\int_{\mathbb{R}^3}E_{\delta Q}\delta Q{\rm d}x}_{N_{11}}
    \underbrace{-aK\Gamma\|\delta Q\|_{L^2}^2}_{N_{12}}
    \underbrace{+bK\Gamma\int_{\mathbb{R}^3}{\rm tr}\left(\delta QQ_1\delta Q+Q_2(\delta Q)^2\right){\rm d}x}_{N_{13}}
    \\&&\nonumber
    \underbrace{-cK\Gamma\int_{\mathbb{R}^3}{\rm tr}(Q_1)^2|\delta Q|^2{\rm d}x}_{N_{14}}
    \underbrace{-cK\Gamma\int_{\mathbb{R}^3}{\rm tr}(Q_2\delta Q){\rm tr}(Q_1\delta Q+\delta Q Q_2){\rm d}x}_{N_{15}}
    \\&&\nonumber
    \underbrace{-\int_{\mathbb{R}^3}(u_2\nabla\delta u+\delta u\nabla u_2)\delta u{\rm d}x}_{N_{16}}
    \underbrace{-\int_{\mathbb{R}^3}(\delta Q(M_{Q_2}+E_{Q_2})-(M_{Q_2}+E_{Q_2})\delta Q)\nabla\delta u{\rm d}x}_{N_{17}}
    \\&&
    \underbrace{+\int_{\mathbb{R}^3}\left(\frac{\partial \mathcal{F}_{LG}[\delta Q]}{\partial \nabla\delta Q}\odot\nabla Q_2\right)\nabla\delta u{\rm d}x+\int_{\mathbb{R}^3}\left(\frac{\partial \mathcal{F}_{LG}[Q_2]}{\partial \nabla\delta Q_2}\odot\nabla\delta Q\right)\nabla\delta u{\rm d}x}_{N_{18}}.
\end{eqnarray}

Firstly, we get the dissipation term of $Q$ with second order derivative in space that
\begin{eqnarray}
    \nonumber
    \|M_{\delta Q}+E_{\delta Q}\|_{L^2}^2&=&L_1^2\|\Delta\delta Q\|_{L^2}^2+\int_{\mathbb{R}^3}{\rm tr}(M_{\delta Q}+E_{\delta Q}-L_1\Delta\delta Q)^2{\rm d}x
    \\&& \nonumber
    +2L_1\int_{\mathbb{R}^3}\Delta Q:(M_{\delta Q}+E_{\delta Q}-L_1\Delta\delta Q){\rm d}x
    \\&\geq& \nonumber
    L_1^2\|\Delta\delta Q\|_{L^2}^2+2L_1(L_2+L_3)\|\nabla{\rm div}\delta Q\|_{L^2}^2-C\int_{\mathbb{R}^3}|\Delta\delta Q||\nabla\delta Q|{\rm d}x
    \\&\geq&
    \frac{L_1^2}{2}\|\Delta\delta Q\|_{L^2}^2-C\|\nabla\delta Q\|_{L^2}^2.
\end{eqnarray}
Then using the H\"older inequality, the interpolation inequalities, and the Cauchy inequality, we have
\begin{eqnarray*}
    \nonumber
    N_{1,1}&\leq& C\int_{\mathbb{R}^3}|\delta u||\nabla Q_2|(|\nabla^2\delta Q|+|\nabla\delta Q|){\rm d}x
    \\&\leq&
    C\|\delta u\|_{L^3}\|\nabla Q_2\|_{L^6}(\|\Delta\delta Q\|_{L^2}+\|\nabla\delta Q\|_{L^2})
    \\&\leq&
    C\|\delta u\|_{L^2}^{\frac12}\|\nabla\delta u\|_{L^2}^{\frac12}(\|\Delta\delta Q\|_{L^2}+\|\nabla\delta Q\|_{L^2})
    \\&\leq&
    \varepsilon\|\Delta\delta Q\|_{L^2}^2+\varepsilon\|\nabla\delta u\|_{L^2}^2+C(\varepsilon)(\|\delta u\|_{L^2}^2+\|\nabla\delta Q\|_{L^2}^2),
    \\ \nonumber
    N_{1,2}&\leq& C\int_{\mathbb{R}^3}|u_2||\nabla\delta Q|(|\nabla^2\delta Q|+|\nabla\delta Q|){\rm d}x
    \\&\leq&
    C\|u_2\|_{L^6}\|\nabla\delta Q\|_{L^3}(\|\Delta\delta Q\|_{L^2}+\|\nabla\delta Q\|_{L^2})
    \\&\leq&
    C\|\nabla u_2\|_{L^2}\|\nabla\delta Q\|_{L^2}^{\frac12}\|\Delta\delta Q\|_{L^2}^{\frac12}(\|\Delta\delta Q\|_{L^2}+\|\nabla\delta Q\|_{L^2})
    \\&\leq&
    \varepsilon\|\Delta\delta Q\|_{L^2}^2+C(\varepsilon)\|\nabla\delta Q\|_{L^2}^2,
    \\ \nonumber
    N_{1,3}+N_{1,4}&\leq& C\int_{\mathbb{R}^3}|\delta Q||\nabla u_2|(|\nabla^2\delta Q|+|\nabla\delta Q|){\rm d}x
    \\&\leq&
    \|\delta Q\|_{L^\infty}\|\nabla u_2\|_{L^2}(\|\Delta\delta Q\|_{L^2}+\|\nabla\delta Q\|_{L^2})
    \\&\leq& \nonumber
    C\|\delta Q\|_{L^2}^{\frac14}\left(\|\delta Q\|_{L^2}^{\frac34}+\|\Delta\delta Q\|_{L^2}^{\frac34}\right)(\|\Delta\delta Q\|_{L^2}+\|\nabla\delta Q\|_{L^2})
    \\&\leq&
    \varepsilon\|\Delta\delta Q\|_{L^2}^2+C(\varepsilon)(\|\delta Q\|_{L^2}^2+\|\nabla\delta Q\|_{L^2}^2),
    \\ \nonumber
    N_2+N_{12}&\leq& |a|\Gamma L_1\|\nabla\delta Q\|_{L^2}^2+|a|\Gamma (L_2+L_3)\|{\rm div}\delta Q\|_{L^2}^2+|a|K\Gamma \|\delta Q\|_{L^2}^2,
    \\
    N_3&\leq& C(\|\delta Q\|_{L^2}^2+\|\nabla\delta Q\|_{L^2}^2),
    \\ \nonumber
    N_4&\leq& C\|\delta Q\|_{L^6}\left(\|Q_1\|_{L^3}+\|Q_2\|_{L^3}\right)(\|\Delta\delta Q\|_{L^2}+\|\nabla\delta Q\|_{L^2})
    \\&\leq&
    \varepsilon\|\Delta\delta Q\|_{L^2}^2+C(\varepsilon)\|\nabla\delta Q\|_{L^2}^2,
    \\ \nonumber
    N_5&\leq& C\|\delta Q\|_{L^6}(\|\Delta\delta Q\|_{L^2}+\|\nabla\delta Q\|_{L^2})\|Q_1\|_{L^6}^2
    \\&\leq&
    \varepsilon\|\Delta\delta Q\|_{L^2}^2+C(\varepsilon)\|\nabla\delta Q\|_{L^2}^2,
    \\ \nonumber
    N_6&\leq&
    C\|Q_2\|_{L^6}(\|\Delta\delta Q\|_{L^2}+\|\nabla\delta Q\|_{L^2})\|Q_1\|_{L^6}\|\delta Q\|_{L^6}
    \\&&
    +C\|Q_2\|_{L^6}^2(\|\Delta\delta Q\|_{L^2}+\|\nabla\delta Q\|_{L^2})\|\delta Q\|_{L^6}
    \\&\leq&
    \varepsilon\|\Delta\delta Q\|_{L^2}^2+C(\varepsilon)\|\nabla\delta Q\|_{L^2}^2,
    \\ \nonumber
    N_{7}&\leq& C\|\delta u\|_{L^2}\|\nabla Q_2\|_{L^3}\|\delta Q\|_{L^6}
    \leq C(\|\delta u\|_{L^2}^2+\|\nabla\delta Q\|_{L^2}^2),
    \\ \nonumber
    N_8&\leq& C\|Q_2\|_{L^3}\|\nabla\delta u\|_{L^2}\|\delta Q\|_{L^6}
    \leq
    \varepsilon\|\nabla\delta u\|_{L^2}^2+C(\varepsilon)\|\nabla\delta Q\|_{L^2}^2,
    \\
    N_9+N_{10}+N_{14}&\leq&0,
    \\
    N_{11}&\leq& C(\|\delta Q\|_{L^2}^2+\|\nabla\delta Q\|_{L^2}^2),
    \\
    N_{13}&\leq& C\left(\|Q_1\|_{L^6}+\|Q_2\|_{L^6}\right)\|\delta Q\|_{L^3}\|\delta Q\|_{L^2}
    \\&\leq&
    C\|\delta Q\|_{L^2}^{\frac32}\|\nabla\delta Q\|_{L^2}^{\frac12}\leq C\left(\|\delta Q\|_{L^2}^2+\|\nabla\delta Q\|_{L^2}^2\right),
    \\ \nonumber
    N_{15}&\leq&
    C\|Q_2\|_{L^6}\|Q_1\|_{L^2}\|\delta Q\|_{L^6}^2+\|Q_2\|_{L^\infty}^2\|\delta Q\|_{L^2}^2
    \\&\leq&
    C\left(\|\delta Q\|_{L^2}^2+\|\nabla\delta Q\|_{L^2}^2\right),
    \\ \nonumber
    N_{16}&\leq&
    C\|u_2\|_{L^6}\|\nabla\delta u\|_{L^2}\|\delta u\|_{L^3}+C\|\nabla u_2\|_{L^2}\|\delta u\|_{L^3}\|\delta u\|_{L^6}
    \\&\leq&
    C\|\delta u\|_{L^2}^{\frac12}\|\nabla\delta u\|_{L^2}^{\frac32}
    \leq
    \varepsilon\|\nabla\delta u\|_{L^2}^2+C(\varepsilon)\|\delta u\|_{L^2}^2,
    \\ \nonumber
    N_{17}&\leq& C\|\delta Q\|_{L^\infty}\left(\|\Delta Q_2\|_{L^2}+\|\nabla Q_2\|_{L^2}\right)\|\nabla\delta u\|_{L^2}
    \\&\leq& \nonumber
    C\|\delta Q\|_{L^2}^{\frac14}\|\Delta\delta Q\|_{L^2}^{\frac34}\|\nabla\delta u\|_{L^2}
    \\&\leq&
    \varepsilon\|\nabla\delta u\|_{L^2}^2+\varepsilon\|\Delta\delta Q\|_{L^2}^2+C(\varepsilon)\|\delta Q\|_{L^2}^2,
    \\ \nonumber
    N_{18}&\leq& C\left(\|\nabla\delta Q\|_{L^3}+\|\delta Q\|_{L^3}\right)\left(\|\nabla Q_2\|_{L^6}+\|Q_2\|_{L^6}\right)\|\nabla\delta u\|_{L^2}
    \\&\leq& \nonumber
    C\left(\|\nabla\delta Q\|_{L^2}^{\frac12}\|\Delta\delta Q\|_{L^2}^{\frac12}+\|\delta Q\|_{L^2}^{\frac12}\|\nabla\delta Q\|_{L^2}^{\frac12}\right)\|\nabla\delta u\|_{L^2}
    \\&\leq&
    \varepsilon\|\nabla\delta u\|_{L^2}^2+\varepsilon\|\Delta\delta Q\|_{L^2}^2+C(\varepsilon)\left(\|\nabla\delta Q\|_{L^2}^2+\|\delta Q\|_{L^2}^2\right).
\end{eqnarray*}

Finally, choosing $\varepsilon$ small enough, we have
\begin{eqnarray}
    \nonumber
    &&\frac12\frac{d}{dt}\int_{\mathbb{R}^3}\left(K|\delta Q|^2+L_1|\nabla\delta Q|^2+(L_2+L_3)|{\rm div}\delta Q|^2+L_4e_{l\alpha k}\delta Q_{l\beta}\delta Q_{\alpha\beta,k}\right){\rm d}x
    \\&& \nonumber
    +\mu\int_{\mathbb{R}^3}|\nabla\delta u|^2{\rm d}x+\frac{\Gamma L_1^2}{2}\int_{\mathbb{R}^3}|\Delta\delta Q|^2{\rm d}x
    \\&\leq&
    C\left(\|\delta u\|_{L^2}^2+\|\delta Q\|_{L^2}^2+\|\nabla\delta Q\|_{L^2}^2\right).
\end{eqnarray}

Therefore, Theorem \ref{th:1.2} follows from this inequality, the Cauchy inequality and the Gronwall inequality if we choose $K>0$ large enough. \qed

\section*{Acknowledgements}This work is supported by the National
Basic Research Program of China (973 Program) (No.2011CB808002), the National Natural
Science Foundation of China (No.11071086, No.11371152, No.11401439 and No.11128102), the Natural Science Foundation of Guangdong Province (No.S2012010010408), the Foundation for Distinguished Young Talents in Higher Education of Guangdong, China (No. 2014KQNCX162), the University Special Research Foundation for Ph.D Program (No.20104407110002), and the Science Foundation for Young Teachers of Wuyi University (No. 2014zk06). The authors would like to thank the anonymous referees for their constructive and interesting suggestions. The authors are also grateful to Prof. Wei Wang for constructive and
helpful discussions while he visited SCNU in April, 2014.


\end{document}